\documentclass[12pt]{amsart}  

\usepackage[latin1]{inputenc}
\usepackage{amsmath} 
\usepackage{amsfonts}
\usepackage{amssymb}
\usepackage{stmaryrd}
\usepackage{latexsym} 
\usepackage{graphicx}
\usepackage{subfigure}
\usepackage{color}
\usepackage{hyperref}
\usepackage{verbatim}
\usepackage[all]{xy}
\usepackage{graphics}
\usepackage{pdfsync}
\usepackage{tikz}
\usepackage{url}
\usepackage{enumerate}

\theoremstyle{plain}
\newtheorem{theorem}{Theorem}[section]
\newtheorem*{theorem*}{Theorem}
\newtheorem{proposition}[theorem]{Proposition}

\newtheorem{lemma}[theorem]{Lemma}
\newtheorem{corollary}[theorem]{Corollary}

\newtheorem{definition}[theorem]{Definition}
\newtheorem{example}[theorem]{Example}

\theoremstyle{remark}

\numberwithin{equation}{section}
\newcommand{\bQ}{\mathbb{Q}}

\newcommand{\suchthat}{\,:\,}
\newcommand{\spam}{\operatorname{span}}

\newcommand{\Sym}{\ensuremath{\operatorname{Sym}}}
\newcommand{\fS}{{\mathfrak S}}

\newcommand{\NCSym}{\ensuremath{\operatorname{NCSym}}}
\newcommand{\slashp}{\mid}
\newcommand{\minel}{\hat{0}_n}
\newcommand{\maxel}{\hat{1}_n}
\newcommand{\mupi}{\mu_\Pi}
\newcommand{\mul}{\mu_L}
\newcommand{\bx}{\text{\bf x}}

\newlength\cellsize 
\savebox2{%
\begin{picture}(15,15)
\put(0,0){\line(1,0){15}}
\put(0,0){\line(0,1){15}}
\put(15,0){\line(0,1){15}}
\put(0,15){\line(1,0){15}}
\end{picture}}
\newcommand\cellify[1]{\def\thearg{#1}\def\nothing{}%
\ifx\thearg\nothing
\vrule width0pt height\cellsize depth0pt\else
\hbox to 0pt{\usebox2\hss}\fi%
\vbox to 15\unitlength{
\vss
\hbox to 15\unitlength{\hss$#1$\hss}
\vss}}
\newcommand\tableau[1]{\vtop{\let\\=\cr
\setlength\baselineskip{-16000pt}
\setlength\lineskiplimit{16000pt}
\setlength\lineskip{0pt}
\halign{&\cellify{##}\cr#1\crcr}}}
\savebox3{%
\begin{picture}(15,15)
\put(0,0){\line(1,0){15}}
\put(0,0){\line(0,1){15}}
\put(15,0){\line(0,1){15}}
\put(0,15){\line(1,0){15}}
\end{picture}}
\newcommand\expath[1]{%
\hbox to 0pt{\usebox3\hss}%
\vbox to 15\unitlength{
\vss
\hbox to 15\unitlength{\hss$#1$\hss}
\vss}}
\newcommand\bas[1]{\omit \vbox to \cellsize{ \vss \hbox to \cellsize{\hss$#1$\hss} \vss}}

\begin{document}

\title[noncommuting chromatic functions revisited]{Chromatic symmetric functions in noncommuting variables revisited}

\author{Samantha Dahlberg}
\address{
School of Mathematical and Statistical Sciences,
Arizona State University,
Tempe AZ 85287-1804, USA}
\email{sdahlber@asu.edu}

\author{Stephanie van Willigenburg}
\address{
 Department of Mathematics,
 University of British Columbia,
 Vancouver BC V6T 1Z2, Canada}
\email{steph@math.ubc.ca}

\thanks{
Both authors were supported  in part by the National Sciences and Engineering Research Council of Canada.}

\begin{abstract} In 1995 Stanley introduced a generalization of the chromatic polynomial of a graph $G$, called the chromatic symmetric function, $X_G$, which was generalized to noncommuting variables, $Y_G$, by Gebhard-Sagan in 2001. Recently there has been a renaissance in the study of $X_G$, in particular in classifying when $X_G$ is a positive linear combination of elementary symmetric or Schur functions. 

We extend this study from $X_G$ to $Y_G$, including establishing the multiplicativity of $Y_G$, and showing $Y_G$ satisfies the $k$-deletion property. Moreover, we completely classify when $Y_G$ is a positive linear combination of elementary symmetric functions in noncommuting variables, and similarly for Schur functions in noncommuting variables, in the sense of Bergeron-Hohlweg-Rosas-Zabrocki. We further establish the natural multiplicative generalization of the fundamental theorem of symmetric functions, now in noncommuting variables, and obtain numerous new bases for this algebra  whose generators are chromatic symmetric functions in noncommuting variables. Finally, we show that of all known symmetric functions in noncommuting variables, only all elementary and specified Schur ones can be realized  as $Y_G$ for some $G$.
\end{abstract}

\maketitle
\tableofcontents

\section{Introduction}\label{sec:intro}  In 1912, Birkhoff introduced the chromatic polynomial of a graph $G$ as a means to solve the four colour problem \cite{Birk}. In 1995 Stanley generalized this to the chromatic symmetric function of $G$, $X_G,$ \cite{Stan95}, which not only generalized theorems of the chromatic polynomial such as the Broken Circuit Theorem \cite{Stan95}, but also led to infinitely many new bases of the algebra of  symmetric functions, $\Sym$ \cite{ChovW, lollipop}, a triple-deletion rule \cite{Orellana}, and connections to representation theory and algebraic geometry \cite {MM}. However, much research on $X_G$ has been devoted to proving the Stanley-Stembridge conjecture \cite[Conjecture 5.5]{StanStem}, which when studied in terms of $X_G$ \cite[Conjecture 5.1]{Stan95} says that if a poset is $(3+1)$-free, then the chromatic symmetric function of its incomparability graph is a positive linear combination of elementary symmetric functions. The elementary symmetric functions arise in a variety of contexts, and one of the best known is the fundamental theorem of symmetric functions, which roughly states that  $\Sym$ is generated by the elementary symmetric functions indexed by positive integers. While the $(3+1)$-conjecture still stands, there has been much progress made towards it, for example, \cite{ChoHuh, Dladders, lollipop, Foley, FoleyKin, Gash, GebSag, GP, Hamel, HuhNamYoo, Tsujie, Wolfe}, and the related question of when $X_G$ is a positive linear combination of Schur functions \cite{Gasharov, Paw, SW, Stanley2}. Most of these results have been achieved by working directly with $X_G$, however, there has been notable success in employing its generalization to quasisymmetric functions \cite{ChoHuh, SW} and symmetric functions in noncommuting variables {$Y_G$} \cite{Dladders, GebSag}. 

The algebra of symmetric functions in noncommuting variables, $\NCSym$, was first studied by Wolf \cite{Wolf}, who aimed to provide an analogue of the fundamental theorem of symmetric functions in this setting. However, little more was done in this area until Rosas and Sagan gave a systematic study of $\NCSym$ in 2004 \cite{RS}, and also give a nice survey of this intervening work. They gave analogues in $\NCSym$ to well-known constructs in $\Sym$ such as Jacobi-Trudi determinants, and the RSK algorithm, and bases analogous to those in $\Sym$ with the exception of the basis of Schur functions. Such an analogue was found by Bergeron, Hohlweg, Rosas and Zabrocki who connected it to the Grothendieck bialgebra of the semi-tower of partition lattice algebras \cite{BHRZ}. Bergeron, Reutenauer, Rosas and Zabrocki furthermore introduced a natural Hopf algebra structure on $\NCSym$ \cite{BRRZ} and the antipode was subsequently computed by Lauve and Mastnak \cite{LauveM}. Bergeron and Zabrocki uncovered further algebraic structure by proving $\NCSym$ was free and cofree \cite{BZ}, and $\NCSym$ was also shown to be isomorphic to the algebra of rook placements by Can and Sagan \cite{CanSagan}. Moreover, $\NCSym$ is connected with the supercharacter theory of all unipotent upper-triangular matrices over a finite field \cite{28authors, Thiem}.

With the role of $\NCSym$ becoming  ever more prominent, it is therefore befitting that the many new results regarding $X_G$ be generalized to $Y_G$ and that the question of $Y_G$ being a positive linear combination of elementary symmetric functions or Schur functions in $\NCSym$ be answered. This paper achieves both of these goals, and uses $Y_G$ to establish a number of new results regarding the algebraic structure of $\NCSym$.

More precisely, this paper is structured as follows. In the next section we recall the relevant definitions. Then in Section~\ref{sec:newtools} we prove the multiplicativity of $Y_G$  in Proposition~\ref{prop:graph_mult}, generalizing Stanley's  \cite[Proposition 2.3]{Stan95}, and use it to establish a multiplicative version in $\NCSym$ of the fundamental theorem of symmetric functions in Theorem~\ref{the:fundamentaltheorem}. We also give a formula for $Y_G$ in terms of a M\"obius function  in Theorem~\ref{the:MobiusStan}, generalizing Stanley's  \cite[Theorem 2.6]{Stan95}. We also show in Proposition~\ref{prop:kdelYG} that $Y_G$ exhibits the triple-deletion property of $X_G$ proved by Orellana and Scott \cite[Theorem 3.1]{Orellana} and its generalization, the $k$-deletion property, proved by the authors \cite[Proposition 5]{lollipop}. In Section~\ref{sec:xposepos} we classify when $Y_G$ is a positive (or negative) linear combination of elementary symmetric functions or Schur functions in $\NCSym$ in Theorems~\ref{the:epos} and ~\ref{the:xpos}, respectively. In particular we show that $Y_G$ is \emph{always} a positive or negative linear combination in the latter case, and use this to show the same is true for all elementary symmetric functions in $\NCSym$ in Corollary~\ref{cor:eisxpos}. Lastly, in Section~\ref{sec:NCSymbases} we  show that the $Y_G$ generate new bases for $\NCSym$ in Theorem~\ref{the:independent}, generalizing the result of Cho and the second author \cite[Theorem 5]{ChovW}. We conclude by showing that, with the exception of all elementary symmetric functions in $\NCSym$ and some specified Schur functions in $\NCSym$, $Y_G$ is never a currently known function in $\NCSym$ in Proposition~\ref{prop:YGasS}, Theorem~\ref{the:YGasothers} and Theorem~\ref{the:YGasx}.

\section{Background}\label{sec:background} We begin by giving   necessary definitions and results that will be used throughout our paper.

Let $n$ be a positive integer. Then we say an \emph{integer partition} $\lambda = (\lambda _1, \lambda_2, \ldots , \lambda _{\ell(\lambda)})$ of $n$ is an unordered list of positive integers whose sum is $n$, and denote this by $\lambda\vdash n$. We call the $\lambda _i$ for $1\leq i \leq \ell(\lambda)$ the \emph{parts} of $\lambda$, call $\ell(\lambda)$ the \emph{length} of $\lambda$ and list the parts in weakly decreasing order. For example, $\lambda = (3,2,2,1) \vdash 8$ and $\ell(\lambda)=4$. We also write $\lambda = (1^{m_1}, 2^{m_2}, \ldots , n^{m_n})$ to indicate that $i$ appears in $\lambda$ $m _i$ times for $1\leq i \leq n$. For example, our previous $\lambda$ can be written as $\lambda = (1^1, 2^2, 3^1, 4^0, 5^0, 6^0, 7^0, 8^0)$. With this in mind, we define $\lambda ! = \lambda _1!\lambda _2! \cdots \lambda _{\ell(\lambda)}!$ and $\lambda ^! = m _1!m _2! \cdots m _n!$ similarly. For example, $(3,2,2,1)! = 3!2!2!1! = 24$ and $(3,2,2,1)^!=1!2!1!0!0!0!0!0!=2$.

Let $[n]=\{1,2,\ldots , n\}$. Then we say a \emph{set partition} $\pi$ of $[n]$ is a family of disjoint  {non-empty} sets $B_1, B_2, \ldots , B_{\ell(\pi)}$ whose union is $[n]$, and denote this by 
$$\pi = B_1/B_2/\cdots / B_{\ell(\pi)} \vdash [n].$$We call the $B_i$ for $1\leq i \leq \ell(\pi)$ the \emph{blocks} of $\pi$, call $\ell(\pi)$ the \emph{length} of $\pi$ and list the blocks by increasing least element. For ease of notation we usually omit the set parentheses and commas of the blocks. For example, if $\pi$ is the family of disjoint  sets $\{1,3,4\}, \{2,5\}, \{6\}, \{7,8\}$ then we write
$$\pi = 134/25/6/78 \vdash [8]$$and $\ell(\pi)=4$. Note that every set partition $\pi \vdash [n]$ determines an integer partition $\lambda \vdash n$ by
$$\lambda (\pi) = \lambda (B_1/B_2/\cdots / B_{\ell(\pi)})= (|B_1|, |B_2|, \ldots , |B_{\ell(\pi)}|)$$listed in weakly decreasing order. For example, $\lambda (134/25/6/78)= (3,2,2,1)$. For a finite set of integers, $S$, define $S+n = \{s+n\suchthat s\in S\}$. Then for two set  partitions $\pi \vdash [n]$ and $\sigma = B_1/B_2/\cdots / B_{\ell(\sigma)} \vdash [m]$ we define their \emph{slash product} to be
$$\pi \slashp \sigma = \pi/(B_1+n)/(B_2+n)/\cdots /(B_{\ell(\sigma)}+n) \vdash [n+m].$$ 

\begin{example}\label{ex:slashp} If $\pi = 134/25 \vdash [5]$ and $\sigma = 1/23 \vdash [3]$ then $\pi \slashp \sigma = 134/25/6/78 \vdash [8]$.
\end{example}

We say a set partition {$\pi$} is \emph{atomic} if there do not exist two non-empty set partitions $\sigma _1, \sigma _2$ such that $\pi = \sigma _1 \slashp \sigma _2$. It is not hard to see that given any set partition $\pi$, it can be written uniquely as
$$\pi=\alpha _1 \slashp \alpha _2 \slashp \cdots \slashp \alpha _k$$such that each $\alpha_i$ is non-empty and atomic, which we will call the \emph{atomic decomposition} of $\pi$.

\begin{example}\label{ex:atomic} The atomic decomposition of $\pi = 134/25/6/78 = 134/25 \slashp 1 \slashp 12.$ Note that for ease of readability larger spaces have been intentionally inserted around the slash product symbols.
\end{example}

The set partitions of $[n]$ are the elements of the \emph{partition lattice} $\Pi _n$ ordered by \emph{refinement}, namely for set partitions $\pi, \sigma \vdash [n]$ we have $\pi \leq \sigma$ if and only if every block of $\pi$ is contained in some block of $\sigma$. For example, $134/25/6/78 < 1346/25/78$. The partition lattice $\Pi _n$ has rank function $r(\pi) = n-\ell(\pi)$ and respective minimal and maximal elements
\begin{align*}
\minel&= 1/2/\cdots/n\\
\maxel&=12\cdots n.
\end{align*}

Its M\"obius function is known and satisfies
\begin{equation}\label{eq:0to1}
\mupi(\minel,\maxel)=(-1)^{n-1}(n-1)!
\end{equation}
and
\begin{equation}\label{eq:0topi}
\mupi(\minel,\pi)=\prod _i (-1)^{\lambda _i-1}(\lambda _i-1)!
\end{equation}
where $\lambda = \lambda(\pi)$.

We will now use set partitions to define the \emph{algebra of symmetric functions in noncommuting variables} $x_1, x_2, \ldots $
$$\NCSym \subset \bQ \ll x_1, x_2, \ldots \gg$$that is a graded algebra
$$\NCSym  = \NCSym ^0 \oplus \NCSym ^1 \oplus \cdots$$where $\NCSym ^0 = \spam \{1\}$. {The} $n$-th graded piece for $n\geq 1$ has the following bases
\begin{align*}
\NCSym ^n&= \spam\{ m_\pi \suchthat \pi\vdash [n]\} = \spam\{ p_\pi \suchthat \pi\vdash [n]\}\\
&= \spam\{ e_\pi \suchthat \pi\vdash [n]\} = \spam\{  {\bx_\pi} \suchthat \pi\vdash [n]\}\\
\end{align*}where these functions are defined as follows.

The \emph{monomial symmetric function in $\NCSym$}, $m_\pi$ where $\pi \vdash [n]$, is given by 
$$m_\pi = \sum _{(i_1, i_2, \ldots , i_n)} x_{i_1}x_{i_2} \cdots x_{i_n}$$summed over all tuples $(i_1, i_2, \ldots , i_n)$ with $i_j=i_k$ if and only if $j$ and $k$ are in the same block of $\pi$.

\begin{example}\label{ex:mpi}
$m_{13/2}=x_1x_2x_1+x_2x_1x_2+x_1x_3x_1+x_3x_1x_3+x_2x_3x_2+x_3x_2x_3+\cdots$
\end{example}

Meanwhile, the \emph{power sum symmetric function in $\NCSym$}, $p_\pi$ where $\pi \vdash [n]$, is given by 
$$p_\pi = \sum _{(i_1, i_2, \ldots , i_n)} x_{i_1}x_{i_2} \cdots x_{i_n}$$summed over all tuples $(i_1, i_2, \ldots , i_n)$ with $i_j=i_k$ if  $j$ and $k$ are in the same block of $\pi$.

\begin{example}\label{ex:ppi}
$p_{13/2}=x_1x_2x_1+x_2x_1x_2+\cdots + x_1^3+x_2^3 +\cdots $
\end{example}

Similarly,  the \emph{elementary symmetric function in $\NCSym$}, $e_\pi$ where $\pi \vdash [n]$, is given by 
$$e_\pi = \sum _{(i_1, i_2, \ldots , i_n)} x_{i_1}x_{i_2} \cdots x_{i_n}$$summed over all tuples $(i_1, i_2, \ldots , i_n)$ with $i_j\neq i_k$ if  $j$ and $k$ are in the same block of $\pi$.

\begin{example}\label{ex:epi}
$e_{13/2}= {x_1x_1x_2+x_1x_2x_2+x_2x_2x_1+x_2x_1x_1}+\cdots + x_1x_2x_3+x_2x_3x_4 +\cdots $
\end{example}

Our final basis, for now, was defined by Bergeron, Hohlweg, Rosas and Zabrocki \cite{BHRZ} to answer Rosas and Sagan's question \cite[Section 9]{RS} of whether there exists a basis for $\NCSym$ that reflects properties of Schur functions. Motivated by representation theory their \emph{Schur function in $\NCSym$}, $\bx_\pi$ where $\pi \vdash [n]$, is given by
\begin{equation}\label{eq:xasp}
\bx_\pi = \sum _{\sigma \leq \pi} \mupi (\sigma, \pi) p _\sigma
\end{equation}
with
\begin{equation}\label{eq:pasx}
p_\pi = \sum _{\sigma \leq \pi} \bx _\sigma .
\end{equation}

\begin{example}\label{ex:pasx}
$$p_{13/2}=\bx _{13/2}+\bx _{1/2/3}$$
\end{example}

If we have a basis $\{ b_\pi\} _{\pi\vdash [n], n\geq 1}$ of $\NCSym$ we say that $f=\sum _\pi c_\pi b_\pi$ is \emph{$b$-positive} if and only if $c_\pi \geq 0$ for all $\pi$, and is \emph{$b$-negative} if and only if $c_\pi \leq 0$ for all $\pi$.
We also note that the names for our bases were chosen as the literature developed because of the \emph{projection map}
$$\rho \suchthat \bQ\ll x_1, x_2, \ldots \gg \longrightarrow \bQ[[x_1, x_2, \ldots ]]$$that lets the variables commute, and in particular takes $\NCSym$ to $\Sym$, with its analogous bases. In particular it takes the monomial symmetric functions, power sum symmetric functions and elementary symmetric functions in $\NCSym$ to their respective scalar multiples in $\Sym$ \cite[Theorem 2.1]{RS}. 

The final concepts that we will need in order to define our objects of study come from graph theory. Let $G$ be a graph with vertices $V(G)$ and edges $E(G)$ and for our purposes we will always assume that $G$ is \emph{finite} and \emph{simple}. Furthermore, we will assume that $G$ is \emph{labelled}, namely if $G$ has $n$ vertices then the vertices are labelled distinctly with $1,2, \ldots ,n$, which we term \emph{$G$ with distinct vertex labels in $[n]$}. Two types of graph that will be particularly useful to us are \emph{trees}, which are connected graphs containing no cycles whose degree 1 vertices are each called a \emph{leaf}, and the complete graphs $K_n$ for $n\geq 1$ that consist of $n$ vertices, every pair of which are adjacent.

Furthermore, given $\pi = B_1/B_2/\cdots / B_{\ell(\pi)} \vdash [n]$  we define
$$K_\pi = K_{|B_1|}\cup {K_{|B_2|}} \cup \cdots \cup K_{|B_{\ell(\pi)}|}$$where $\cup$ denotes disjoint union of graphs, and the labels on $K_{|B_i|}$ are the elements of the block $B_i$.

\begin{example}\label{ex:Kpi} $K_{134/25/6/78}$ is the following graph.

$$
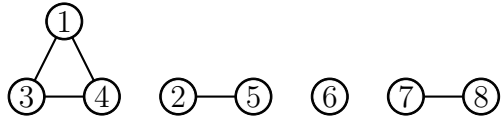
\begin{tikzpicture}
\coordinate (A) at (0,0);
\coordinate (B) at (1,0);
\coordinate (C) at (0.5,1);
\draw[thick] (A)--(B)--(C)--(A);
\filldraw (A) circle (7pt);
\filldraw (B) circle (7pt);
\filldraw (C) circle (7pt);
\filldraw[white] (A) circle [radius=6pt] node {\textcolor{black}{$3$}};
\filldraw[white] (B) circle [radius=6pt] node {\textcolor{black}{$4$}};
\filldraw[white] (C) circle [radius=6pt] node {\textcolor{black}{$1$}};
\end{tikzpicture} 
\hspace{.2in}
\begin{tikzpicture}
\coordinate (A) at (0,0);
\coordinate (B) at (1,0);
\draw[thick] (A)--(B);
\filldraw (A) circle (7pt);
\filldraw (B) circle (7pt);
\filldraw[white] (A) circle [radius=6pt] node {\textcolor{black}{$2$}};
\filldraw[white] (B) circle [radius=6pt] node {\textcolor{black}{$5$}};
\end{tikzpicture} 
\hspace{.2in}
\begin{tikzpicture}
\coordinate (A) at (0,0);
\filldraw (A) circle (7pt);
\filldraw[white] (A) circle [radius=6pt] node {\textcolor{black}{$6$}};
\end{tikzpicture}
\hspace{.2in}
\begin{tikzpicture}
\coordinate (A) at (0,0);
\coordinate (B) at (1,0);
\draw[thick] (A)--(B);
\filldraw (A) circle (7pt);
\filldraw (B) circle (7pt);
\filldraw[white] (A) circle [radius=6pt] node {\textcolor{black}{$7$}};
\filldraw[white] (B) circle [radius=6pt] node {\textcolor{black}{$8$}};
\end{tikzpicture}$$
\end{example}

Two other tools that will be useful will be those of deletion and contraction. For the first of these, if $G$ is a graph and $\epsilon \in E(G)$, then $G-\epsilon$ denotes $G$ with $\epsilon$ deleted, and more generally if $S\subseteq E(G)$, then $G-S$ denotes $G$ with every $\epsilon \in S$ deleted. Meanwhile for the second of these, if $G$ and $\epsilon$ are as before, then $G/\epsilon$ denotes $G$ with $\epsilon$ contracted and its vertices at either end identified. However, what will be the apex of our attention will be the chromatic symmetric function in $\NCSym$, which we will define after one final notion. Given a graph $G$, we define a \emph{proper colouring} $\kappa$ of $G$ to be a function
$$\kappa \suchthat V(G) \longrightarrow \{1,2,\ldots \}$$such that if  {$u, v \in V(G)$} are adjacent, then  {$\kappa(u)\neq \kappa (v)$.}

\begin{definition}\label{def:YG} \cite[Definition 3.1]{GebSag} For a graph $G$ with distinct vertex labels in $[n]$ let the label $i$ be on vertex $v_i$. Then the \emph{chromatic symmetric function in $\NCSym$} is defined to be
$$Y_G=\sum _\kappa x_{\kappa(v_1)}x_{\kappa(v_2)}\cdots x_{\kappa(v_n)}$$where the sum is over all proper colourings $\kappa$ of $G$. If $G$ is the empty graph then $Y_G=1$.
\end{definition}

\begin{example}\label{ex:YG} If $G=K_{13/2}$, then our graph is
$$
\begin{tikzpicture}
\coordinate (A) at (0,0);
\coordinate (B) at (1,0);
\draw[thick] (A)--(B);
\filldraw (A) circle (7pt);
\filldraw (B) circle (7pt);
\filldraw[white] (A) circle [radius=6pt] node {\textcolor{black}{$1$}};
\filldraw[white] (B) circle [radius=6pt] node {\textcolor{black}{$3$}};
\end{tikzpicture}
\hspace{.2in}
\begin{tikzpicture}
\coordinate (A) at (0,0);
\filldraw (A) circle (7pt);
\filldraw[white] (A) circle [radius=6pt] node {\textcolor{black}{$2$}};
\end{tikzpicture}$$and $Y_{K_{13/2}} = {x_1x_1x_2+x_1x_2x_2+x_2x_2x_1+x_2x_1x_1}+\cdots + x_1x_2x_3+x_2x_3x_4 +\cdots = e_{13/2}$.
\end{example}

We conclude this section by recalling 
$$\rho (Y_G) = X_G$$where $X_G$ is the chromatic symmetric function originally defined by Stanley \cite{Stan95} and motivated the definition and study of its analogue $Y_G$ by Gebhard and Sagan \cite{GebSag}.

\section{New tools for chromatic symmetric functions in $\NCSym$}\label{sec:newtools} In this section we generalize some of the classical and contemporary results for the chromatic symmetric function in $\Sym$ to that in $\NCSym$. More precisely, given a graph $G$ we derive   $Y_G$ in terms of power sum symmetric functions in $\NCSym$ whose coefficients are obtained from the M\"obius function of the lattice of contractions of $G$, generalizing \cite[Theorem 2.6]{Stan95}. We  show that $Y_G$ is multiplicative, generalizing \cite[Proposition 2.3]{Stan95}. We also show that $Y_G$ satisfies the properties of $X_G$ known as triple-deletion \cite[Theorem 3.1]{Orellana} and more generally $k$-deletion \cite[Proposition 5]{lollipop}.

For the first result, given a graph $G$ with distinct vertex labels in $[n]$, we say $\pi = B_1/B_2/\cdots/B_{\ell(\pi)} \vdash [n]$ is a \emph{connected partition} of $G$ if the graph induced by the vertices whose labels belong to any block $B_i$ is connected. Let the \emph{lattice of contractions}, $L_G$, of $G$ be the set of all connected partitions of $G$ partially ordered by refinement, and let $\mul$ be its M\"obius function. Then we have the following generalization of \cite[Theorem 2.6]{Stan95},  {whose statement and proof are analogous to the original.}

\begin{theorem}\label{the:MobiusStan}For a graph $G$ with distinct vertex labels in $[n]$ we have for any $\pi\in L_G$ that  $\mu_L(\minel, \pi)\neq 0$ and 
$$Y_G=\sum_{\pi\in L_G}\mu_L(\minel, \pi)p_{\pi}.$$
\end{theorem}

\begin{proof}
{By \cite[Equation (1)]{Stan95} we have that $|\mu_L(\minel, \pi)|> 0$ and hence $\mu_L(\minel, \pi)\neq 0$.}
For a set partition $\pi \in L_G$ define $$Y_{\pi}=\sum_{\kappa}x_{\kappa}$$
to be the sum over all special colourings, $\kappa$,   {of all the vertices that for $u,v\in V(G)$ is} given by: (i) 
 if $u$ and $v$ are in the same block of $\pi$ then $\kappa(u)=\kappa(v)$, and (ii) if instead $u$ and $v$ are in different blocks and there is an edge between $u$ and $v$, then $\kappa(u)\neq \kappa(v)$. 

Note that \emph{any} colouring $\kappa$ of $G$ contributes uniquely to one $Y_{\pi}$. We can see this by starting with any colouring $\kappa$ and form each block of its partition $\pi$ by colours, so that all vertices of the same colour are in the same block. Then we refine these blocks   further to respect connected components, so that $\pi$ is a connected partition of $G$.

Next using the definition of power sum symmetric functions in $\NCSym$ we have for  {$\sigma \in L_G$} that
$$p_{\sigma}=\sum_{\pi\geq \sigma,  {\pi\in L_G}} Y_{\pi}.$$
By M\"obius inversion we obtain the theorem, since when $\pi= \minel$ the definition of special colouring coincides with that of proper colouring, so $Y_{\minel} = Y_G$.
\end{proof}

This is not the only expression for $Y_G$ in terms of the power sum symmetric functions in $\NCSym$, and it is the next one that will be useful for our second result. {For it,}  given a graph $G$ and edge set $S\subseteq E(G)$ we define $\pi(S)$ to be the set partition whose blocks are determined by the vertex labels of the connected components of $G$ restricted to the edges in $S$. {For example, in Example~\ref{ex:Kpi} if $S=\{(1,3),(1,4),(7,8)\}$ then $\pi(S) = 134/2/5/6/78$.}

\begin{proposition}\cite[Proposition 3.6]{GebSag}\label{prop:gebhard_sagan}
For a graph $G$ we have $$Y_G=\sum_{S\subseteq E(G)}(-1)^{|S|} p_{\pi(S)}.$$
\end{proposition}

We also need the following result on the multiplicativity of the power sum symmetric functions in $\NCSym$.

\begin{lemma}\cite[Lemma 4.1 (i)]{BHRZ}\label{lem:powersum}For two set partitions $\pi$ and $\sigma$ we have
$$p_{\pi \slashp \sigma}=p_{\pi}p_{\sigma}.$$
\end{lemma}

Lastly, given two graphs $G$, with distinct vertex labels in $[n]$, and $H$, with distinct vertex labels in $[m]$, define $G\slashp H$ to be the disjoint union of $G$ and $H$ where the vertices corresponding to $G$ have labels in $[n]$ in the same relative order as $G$, and the vertices corresponding to $H$ have labels $\{n+1,n+2,\ldots ,n+m\}$ in the same relative order as $H$. 

\begin{example}\label{ex:GslashpH} If 
$$\begin{tikzpicture} \draw (-1,0) node {\textcolor{black}{$G =$}};
\coordinate (A) at (0,0);
\coordinate (B) at (1,0);
\coordinate (C) at (0.5,1);
\draw[thick] (A)--(B)--(C)--(A);
\filldraw (A) circle (7pt);
\filldraw (B) circle (7pt);
\filldraw (C) circle (7pt);
\filldraw[white] (A) circle [radius=6pt] node {\textcolor{black}{$3$}};
\filldraw[white] (B) circle [radius=6pt] node {\textcolor{black}{$4$}};
\filldraw[white] (C) circle [radius=6pt] node {\textcolor{black}{$1$}};
\end{tikzpicture}
\hspace{.2in}
\begin{tikzpicture}
\coordinate (A) at (0,0);
\coordinate (B) at (1,0);
\draw[thick] (A)--(B);
\filldraw (A) circle (7pt);
\filldraw (B) circle (7pt);
\filldraw[white] (A) circle [radius=6pt] node {\textcolor{black}{$2$}};
\filldraw[white] (B) circle [radius=6pt] node {\textcolor{black}{$5$}};
\end{tikzpicture} 
\hspace{.2in}
\begin{tikzpicture} \draw (-1,0) node {\textcolor{black}{$H =$}};
\coordinate (A) at (0,0);
\filldraw (A) circle (7pt);
\filldraw[white] (A) circle [radius=6pt] node {\textcolor{black}{$1$}};
\end{tikzpicture}
\hspace{.2in}
\begin{tikzpicture}
\coordinate (A) at (0,0);
\coordinate (B) at (1,0);
\draw[thick] (A)--(B);
\filldraw (A) circle (7pt);
\filldraw (B) circle (7pt);
\filldraw[white] (A) circle [radius=6pt] node {\textcolor{black}{$2$}};
\filldraw[white] (B) circle [radius=6pt] node {\textcolor{black}{$3$}};
\end{tikzpicture}$$then we get the following.
$$
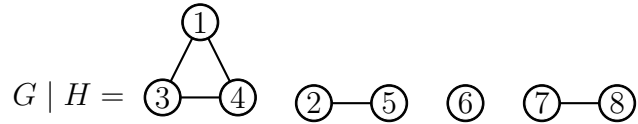
\begin{tikzpicture} \draw (-1.25,0) node {\textcolor{black}{$G\slashp H =$}};
\coordinate (A) at (0,0);
\coordinate (B) at (1,0);
\coordinate (C) at (0.5,1);
\draw[thick] (A)--(B)--(C)--(A);
\filldraw (A) circle (7pt);
\filldraw (B) circle (7pt);
\filldraw (C) circle (7pt);
\filldraw[white] (A) circle [radius=6pt] node {\textcolor{black}{$3$}};
\filldraw[white] (B) circle [radius=6pt] node {\textcolor{black}{$4$}};
\filldraw[white] (C) circle [radius=6pt] node {\textcolor{black}{$1$}};
\end{tikzpicture} 
\hspace{.2in}
\begin{tikzpicture}
\coordinate (A) at (0,0);
\coordinate (B) at (1,0);
\draw[thick] (A)--(B);
\filldraw (A) circle (7pt);
\filldraw (B) circle (7pt);
\filldraw[white] (A) circle [radius=6pt] node {\textcolor{black}{$2$}};
\filldraw[white] (B) circle [radius=6pt] node {\textcolor{black}{$5$}};
\end{tikzpicture} 
\hspace{.2in}
\begin{tikzpicture}
\coordinate (A) at (0,0);
\filldraw (A) circle (7pt);
\filldraw[white] (A) circle [radius=6pt] node {\textcolor{black}{$6$}};
\end{tikzpicture}
\hspace{.2in}
\begin{tikzpicture}
\coordinate (A) at (0,0);
\coordinate (B) at (1,0);
\draw[thick] (A)--(B);
\filldraw (A) circle (7pt);
\filldraw (B) circle (7pt);
\filldraw[white] (A) circle [radius=6pt] node {\textcolor{black}{$7$}};
\filldraw[white] (B) circle [radius=6pt] node {\textcolor{black}{$8$}};
\end{tikzpicture}$$
\end{example}

Then we have the following generalization of \cite[Proposition 2.3]{Stan95}.

\begin{proposition}\label{prop:graph_mult}For two graphs $G$, with distinct vertex labels in $[n]$, and $H$, with distinct vertex labels in $[m]$, we have
$$Y_{G\slashp H}=Y_GY_H.$$
\end{proposition}

\begin{proof}
Using Proposition~\ref{prop:gebhard_sagan} we have 
\begin{align*}
Y_{G\slashp H}&=\sum_{S\subseteq E(G)\cup E(H)}(-1)^{|S|}p_{\pi(S)}\\
&=\sum_{S_1\subseteq E(G), S_2\subseteq E(H)}(-1)^{|S_1|+|S_2|}p_{\pi(S_1\cup S_2)}.
\end{align*}
Note that because of the  vertex labelling of $G\slashp H$ we know that $\pi(S_1\cup S_2)=\pi(S_1)\slashp \pi( S_2)$. By Lemma~\ref{lem:powersum} we have 
\begin{align*}
Y_{G\slashp H}&=\sum_{S_1\subseteq E(G), S_2\subseteq E(H)}(-1)^{|S_1|+|S_2|}p_{\pi(S_1)\slashp \pi( S_2)}\\
&=\sum_{S_1\subseteq E(G), S_2\subseteq E(H)}(-1)^{|S_1|+|S_2|}p_{\pi(S_1)}p_{\pi( S_2)}\\
&=\sum_{S_1\subseteq E(G)}(-1)^{|S_1|} p_{\pi(S_1)}\sum_{S_2\subseteq E(H)}(-1)^{|S_2|}p_{\pi( S_2)}\\
&=Y_{G}Y_H.
\end{align*}\end{proof}

Our final generalization is that the triple-deletion property \cite[Theorem 3.1]{Orellana}, and more generally 
the $k$-deletion property \cite[Proposition 5]{lollipop}, extend to chromatic symmetric functions in $\NCSym$. 

\begin{proposition}\label{prop:kdelYG} Let $G$ be a graph  such that $\epsilon_1,\epsilon_2,\ldots,\epsilon_k\in E(G)$ form a $k$-cycle for $k\geq 3$. Then
$$\sum_{S\subseteq [k-1]}(-1)^{|S|}Y_{G-\cup_{i\in S}\{\epsilon_i\}}=0.$$In particular let $k=3$ and $\epsilon_1,\epsilon_2, \epsilon_3\in E(G)$ form a triangle. Then
$$Y_G=Y_{G-\epsilon_1} + Y_{G-\epsilon_2} - Y_{G-\{\epsilon_1, \epsilon _2\}}.$$
\end{proposition}

\begin{proof}
We assume that $G$ has a $k$-cycle for $k\geq 3$, with edges $\epsilon_1,\epsilon_2,\dots, \epsilon_{k}$ and vertices $v_1,v_2,\dots, v_{k}$ as below.
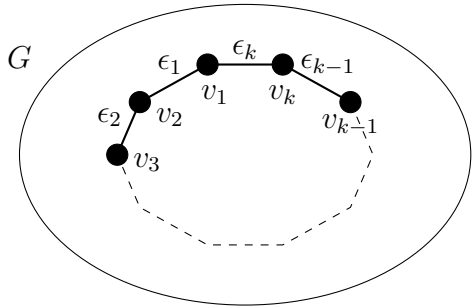
\begin{figure}[h]
\begin{center}
\begin{tikzpicture}[scale=1]
\draw (0,0) ellipse (3 and 2);
\coordinate (1) at  (-1.7,0);
\coordinate (2) at (-1.4,.7);
\coordinate (3) at (-.5,1.2);
\coordinate (4) at (.5,1.2);
\coordinate (5) at (1.4,.7);
\coordinate (6) at  (1.7,0);
\coordinate (7) at (1.4,-.7);
\coordinate (8) at  (.5,-1.2);
\coordinate (9) at  (-.5,-1.2);
\coordinate (10) at  (-1.4,-.7);
\draw[->][black, thick] (1)--(2)--(3)--(4)--(5);
\draw[->][black, dashed] (5)--(6)--(7)--(8)--(9)--(10)--(1);
\draw (0,1.4) node {$\epsilon_{k}$};
\draw (-3,1.3) node {$G$};
\draw (-1,1.2) node {$\epsilon_{1}$};
\draw (-1.8,.5) node {$\epsilon_{2}$};
\draw (1.1,1.2) node {$\epsilon_{k-1}$};
\draw (-.4,.8) node {$v_{1}$};
\draw (.5,.8) node {$v_{k}$};
\draw (-1,.5) node {$v_{2}$};
\draw (-1.3,-.1) node {$v_{3}$};
\draw (1.4,.4) node {$v_{k-1}$};
\filldraw [black] (1) circle (4pt);
\filldraw [black] (2) circle (4pt);
\filldraw [black] (3) circle (4pt);
\filldraw [black] (4) circle (4pt);
\filldraw [black] (5) circle (4pt);
\end{tikzpicture}
\end{center}
\end{figure}

We will prove the formula using a sign-reversing involution without any fixed points. Our signed set will be pairs $(\kappa, S)$ where $S\subseteq[k-1]$ and $\kappa$ is a proper colouring on the graph $G-\cup_{i\in S}\{\epsilon_i\}$. The weight of this pair will be $(-1)^{|S|}x_{\kappa}$, that is, the monomial associated to the colouring $\kappa$ with sign determined by $|S|$. 
Since the edge $\epsilon_k$ is present in all graphs $G-\cup_{i\in S}\{\epsilon_i\}$ for any $S\subseteq [k-1]$  all colourings will have at least two colours on the vertices $v_1,v_2,\dots, v_{k}$. Hence, there will always exist a smallest $j\in [k-1]$ where the two colours on $v_j$ and $v_{j+1}$ are different. We map $(\kappa, S)$ to $(\kappa, T)$ where $T=S\cup\{j\}$ if $j\notin S$ and otherwise $T = S\setminus\{j\}$. The colouring $\kappa$ is a proper colouring on $G-\cup_{i\in S}\{\epsilon_i\}$ and $G-\cup_{i\in T}\{\epsilon_i\}$ since we are only including or excluding the edge $\epsilon_j$, which has adjacent vertices $v_j$ and $v_{j+1}$ with different colours. We can easily see that this is an involution that only switches the sign of the weight, and since there are no  fixed points our desired sum is zero. \end{proof}

\section{The classification of $\bx$-positive and $e$-positive chromatic symmetric functions in $\NCSym$} \label{sec:xposepos} In 2001 Gebhard and Sagan observed that \cite[Section 6]{GebSag}
\begin{quote} Unfortunately, even for some of the simplest graphs, $Y_G$ is usually not $e$-positive. The only graphs that are obviously $e$-positive are the complete graphs on $n$ vertices and their complements.
\end{quote} In this section we confirm their observation by proving in Theorem~\ref{the:epos} that $Y_G$ is $e$-positive if and only if $G$ is a union of complete graphs, and is never $e$-negative. 

A closely related question is to classify when $Y_G$ is $\bx$-positive. Remarkably, it is always either $\bx$-positive or $\bx$-negative, depending on the number of vertices and connected components of $G$, as we shall see in Theorem~\ref{the:xpos}. The elegant resolution to when $Y_G$ is $e$-positive or $\bx$-positive is in stark contrast to the analogous questions for $X_G$ in $\Sym$, which are still   open and the subject of much research, as discussed in the introduction. We will first work towards proving when $Y_G$ is $\bx$-positive or $\bx$-negative, and for this we require some tools.


The first tool is the Relabelling Proposition, which considers how, given a graph $G$, permuting the vertex labels of $G$ affects $Y_G$.  {Given  the symmetric group $\fS_n$, a permutation $\delta\in \fS_n$ and $f\in \NCSym$,} define $\delta\circ f$ to be the function after we permute the placements of the variables, rather than the subscripts. For example, having $\delta = 213$ acting on $m_{1/23}$ means we switch the first two variables so $\delta \circ m_{1/23}=m_{13/2}$. Also define for a graph $G$ with distinct vertex labels in $[n]$ a new graph $\delta(G)$, which is $G$ but we permute the labels of the vertices subject to $\delta$.  {Likewise, define for a set partition $\pi \vdash [n]$  a new set partition $\delta(\pi)$, which is $\pi$ but we permute the block elements subject to $\delta$.} The following is due to Gebhard and Sagan.

\begin{proposition}\emph{(Relabelling Proposition~\cite[Proposition 3.3]{GebSag})} \label{prop:relabeling}For a graph $G$ with distinct vertex labels in $[n]$ and $\delta\in\fS_n$ we have
$$Y_{\delta(G)}=\delta\circ Y_G.$$
\end{proposition}

Consequently, for any graph $G$ with more than one component, using the Relabelling Proposition (Proposition~\ref{prop:relabeling})  we can calculate $Y_G$ using the connected components of $G$.

\begin{corollary} \label{cor:Y_G}
For any graph $G=G_1\cup G_2$ a disjoint union of graphs $G_1$ and $G_2$ with distinct vertex labels in $[n]$, let $\delta\in\fS _n$ be a permutation such that $\delta(G)=G_1\slashp G_2$. Then we have $$Y_G=\delta^{-1}\circ(Y_{G_1}Y_{G_2}).$$
\end{corollary}

The second tool is Deletion-Contraction for $Y_G$ by Gebhard and Sagan. Define the {\it induced} monomial to be 
$$x_{i_1}x_{i_2}\cdots x_{i_{n-1}}\uparrow=x_{i_1}x_{i_2}\cdots x_{i_{n-1}}^2$$
 where we make an extra copy of the last variable at the end and extend this definition linearly. 
Given a set partition $\pi\vdash [n-1]$ we define $\pi\oplus n\vdash [n]$ to be the set partition where we place $n$ in the same block as $n-1$. For example, $14/23\oplus 5=145/23$. Gebhard and Sagan \cite[p 233]{GebSag} state  for $\pi\vdash [n-1]$ that 
\begin{equation}
m_{\pi}\uparrow = m_{\pi\oplus n}\text{ and } p_{\pi}\uparrow = p_{\pi\oplus n}.
\label{eq:induced}
\end{equation}

\begin{proposition}\emph{(Deletion-Contraction~\cite[Proposition 3.5]{GebSag})} \label{prop:DeletionContraction}
For a graph $G$ with distinct vertex labels in $[n]$ and an edge $\epsilon$ between vertices labelled $n$ and $n-1$ we have 
$$Y_G=Y_{G- \epsilon}-Y_{G/\epsilon}\uparrow$$where the new vertex remaining after the contraction of $\epsilon$ is labelled $n-1$.
\end{proposition}

Before we show that all graphs $G$ have $Y_G$ be  $\bx$-positive or $\bx$-negative we will show that all trees $T$ have  $Y_G$ be $\bx$-positive or $\bx$-negative. 

\begin{lemma} \label{lem:trees}
For a tree $T$ with distinct vertex labels in $[n]$   we have  
$$Y_T=(-1)^{n-1}\underset{\text{no leaf is alone}}{\sum_{\sigma\vdash [n], S_{\sigma}=E(T)}}\bx_{\sigma}$$
where $S_{\sigma}$ is the set of all edges in the path from any $u$ to any $v$ in $V(T)$ where $u$ and $v$ are in the same block of $\sigma$ and ``no leaf is alone'' means that if vertex $v$ is a leaf, then the block containing $v$ has size at least two. 
\end{lemma}

Before we prove this lemma we give a small but illustrative example.

\begin{example}\label{ex:YGasx} If 
$$\begin{tikzpicture} \draw (-1,0) node {\textcolor{black}{$G =$}};
\coordinate (A) at (0,0);
\coordinate (B) at (1,0);
\coordinate (C) at (2,0);
\draw[thick] (A)--(B)--(C);
\filldraw (A) circle (7pt);
\filldraw (B) circle (7pt);
\filldraw (C) circle (7pt);
\filldraw[white] (A) circle [radius=6pt] node {\textcolor{black}{$1$}};
\filldraw[white] (B) circle [radius=6pt] node {\textcolor{black}{$2$}};
\filldraw[white] (C) circle [radius=6pt] node {\textcolor{black}{$3$}};
\end{tikzpicture}$$
then $Y_G=(-1)^2(\bx_{123}+ \bx_{13/2}) = \bx_{123}+ \bx_{13/2}$.
\end{example}

\begin{proof}
Let $T$ be a tree on $n$ vertices  with edge set $E$. Without loss of generality say $T$ has a leaf at vertex $n$ connected to vertex $n-1$, by the Relabelling Proposition (Proposition~\ref{prop:relabeling}). Call this edge $\epsilon$ and  $\bar E=E\setminus \{\epsilon\}$. Using Deletion-Contraction (Proposition~\ref{prop:DeletionContraction}), Proposition~\ref{prop:gebhard_sagan}  and Equation~\eqref{eq:pasx} we have 
\begin{align*}
Y_T&=Y_{T-\epsilon}-Y_{T/\epsilon}\uparrow\\
&=\sum_{S\subseteq \bar E}(-1)^{|S|}p_{\pi(S)|1}-\sum_{S\subseteq \bar E}(-1)^{|S|}p_{\pi(S)\oplus n}\\
&=\sum_{S\subseteq \bar E}\sum_{\sigma\leq\pi(S)|1} (-1)^{|S|}\bx_{\sigma}-\sum_{S\subseteq \bar E}\sum_{\sigma\leq\pi(S)\oplus n}(-1)^{|S|}\bx_{\sigma}.
\end{align*}
Between these two sums of sums we cancel out all the $\bx_{\sigma}$ associated to $\sigma\leq \pi(S)|1$ in the first sum with the identical term in the second sum since certainly $\sigma\leq \pi(S)\oplus n$. We then have
\begin{align*}
Y_T
&=-\sum_{S\subseteq \bar E}\underset{\text{$n$ is not alone}}{\sum_{\sigma\leq\pi(S)\oplus n}}(-1)^{|S|}\bx_{\sigma}
\end{align*}
where  ``$n$ is not alone'' in $\sigma\vdash[n]$ means that the block containing $n$ has at least two elements. 
Our next step will be to switch the order of summation. To do this, given a $\sigma$ where $n$ is not alone, we must identify all  $S\subseteq \bar E$ that satisfy $\sigma\leq \pi(S)\oplus n$. 
Consider $S_{\sigma}\subseteq {E}$ the collection of all edges in the path between $u$ and $v$ for all $u$ and $v$ in the same block of $\sigma$. 
 
First we will show that $\sigma\leq \pi(S_{\sigma}\setminus \{\epsilon\})\oplus n$. We will show this by considering any $a,b$ in the same block of $\sigma$ and show $a,b$ are in the same block of $\pi(S_{\sigma}\setminus \{\epsilon\})\oplus n$. Say that $a,b\neq n$ are in the same block of $\sigma$, we know that all edges in the unique path between $a$ and $b$ are in $S_{\sigma}$ and this path avoids the edge $\epsilon$. This guarantees that $a$ and $b$ are in the same block in $\pi(S_{\sigma}\setminus \{\epsilon\})$. {Say that $a\neq n$ and $b=n$ are} in the same block of $\sigma$. Then $S_{\sigma}$ certainly contains $\epsilon$.  {This} guarantees that all edges in the unique path from $n$ to $n-1$ to $a$ are in $S_{\sigma}$. This {then} guarantees that $a$ and $n-1$ are in the same block in $\pi(S_{\sigma}\setminus \{\epsilon\})$, which implies $a$ and $n$ are in the same block in $\pi(S_{\sigma}\setminus \{\epsilon\})\oplus n$. 

Next we will show that $S_{\sigma}\setminus \{\epsilon\} \subseteq S$ if and only if $\sigma \leq \pi(S)\oplus n$, which identifies all the $S\subseteq \bar E$ that we need. Certainly if $S_{\sigma}\setminus \{\epsilon\} \subseteq S$ we have $\sigma\leq \pi(S_{\sigma}\setminus \{\epsilon\})\oplus n \leq \pi(S)\oplus n$. Say instead that $\sigma \leq \pi(S)\oplus n$ and $\bar{uv}\neq \epsilon$ is an edge in $S_{\sigma}$, which we want to show is also in $S$. Because $\bar{uv}\in S_{\sigma}$ we know that $u$ and $v$ are on some path $P$ in $T$ with end points $a$ and $b$ in the same block of $\sigma$. {If} $a,b\neq n$ then this means $a$ and $b$ are in the same block of $\pi(S)$, which only happens when all edges in the path $P$ are in $S$, so $\bar{uv}\in S$. If instead $a\neq n$ and $b=n$, then we still have $a$ and $n$ in the same block of $\pi(S)\oplus n$ meaning that $n-1$ is in the same block as $a$ similarly implying that all edges between $a$ and $n-1$, which includes $\bar{uv}$, are in $S$. 
Switching the order of summation we have 
\begin{align*}
Y_T
&=-\underset{\text{$n$ is not alone}}{\sum_{\sigma}}{\sum_{S_\sigma\setminus\{\epsilon\}\subseteq S\subseteq \bar E}}(-1)^{|S|}\bx_{\sigma}.
\end{align*}
Note that unless $S_\sigma\setminus\{\epsilon\}=\bar E$ the inner sum is 0, so we have 
\begin{align*}
Y_T
&=-\underset{\text{$n$ is not alone}}{\sum_{\sigma, S_\sigma\setminus\{\epsilon\}=\bar E}}(-1)^{|\bar E|}\bx_{\sigma}=(-1)^{n-1}\underset{\text{$n$ is not alone}}{\sum_{\sigma, S_{\sigma}\setminus\{\epsilon\}=\bar E}}\bx_{\sigma}
=(-1)^{n-1}\underset{\text{$n$ is not alone}}{\sum_{\sigma, S_{\sigma}=E}}\bx_{\sigma}
\end{align*}
after additionally using the fact that $T$ is a tree so $|\bar E|=n-2$. The lemma then follows by observing the same proof applies to every leaf.
\end{proof}

We will now show that inducing an $\bx$-positive function preserves $\bx$-positivity. For this we need to recall the following.

\begin{lemma}\label{lem:xprod}\cite[Lemma 4.2 (i)]{BHRZ} For two set partitions $\pi$ and $\sigma$ we have
$$\bx_{\pi \slashp \sigma}=\bx_{\pi}\bx_{\sigma}.$$
\end{lemma}

Also note that given a set partition $\pi\vdash [n]$ and $\delta \in \fS _n$, {since   $\delta\circ p_{\pi}=p_{\delta(\pi)}$ \cite[p 230]{GebSag} we also have} that  
\begin{equation}\label{eq:xdelta}
\delta\circ \bx_{\pi}=\bx_{\delta(\pi)}.
\end{equation}

\begin{lemma}\label{lem:x_induce}
We have $$\bx_{\hat 1 _{n-1}}\uparrow=\sum_{\sigma}\bx_{\sigma}$$ 
summed over all $\sigma\vdash [n]$ with at most two blocks where every block contains $n$ or $n-1$. Further, $\bx_{\pi}\uparrow$ is $\bx$-positive for all $\pi\vdash [n]$ and hence inducing any $\bx$-positive function preserves $\bx$-positivity. 
\end{lemma}

\begin{proof}By change of bases formulas in Equations~\eqref{eq:xasp} and~\eqref{eq:pasx} and the formula for inducing in~\eqref{eq:induced} we have 
\begin{align*}
\bx_{\hat 1 _{n-1}}&=\sum_{\sigma{\vdash [n-1]}}\mu_{\Pi}(\sigma,\hat 1_{n-1})p_{\sigma}\\
\Rightarrow\bx_{\hat 1_{n-1}}\uparrow&
=\sum_{\sigma{\vdash [n-1]}}\mu_{\Pi}(\sigma,\hat 1 _{n-1})p_{\sigma}\uparrow
=\sum_{\sigma{\vdash [n-1]}}\mu_{\Pi}(\sigma,\hat 1 _{n-1})p_{\sigma\oplus  n}
=\sum_{\sigma{\vdash [n-1]}}\sum_{\pi\leq \sigma\oplus n} \mu_{\Pi}(\sigma,\hat 1 _{n-1})\bx_{\pi}.
\end{align*}
{Next note} that $\{\sigma\vdash[n-1]:\pi\leq \sigma\oplus  n\}= {\{\sigma\vdash[n-1]:\tilde \pi\leq \sigma \leq \hat{1} _{n-1}\}}$ 
where $\tilde \pi$ is the set partition $\pi$ where the blocks containing $n$ and $n-1$ are merged and $n$ is removed.  {This is  because for a fixed $\pi\vdash[n]$ the minimal set partition $\bar{\pi}\geq \pi$ where $n-1$ and $n$ are in the same block is $\pi$ but the blocks containing $n$ and $n-1$ are merged. Note that $\sigma\oplus  n$ has $n$ and $n-1$ in the same block so because $\pi\leq \sigma\oplus n$ we have $\bar{\pi}\leq \sigma\oplus n$ and removing the $n$ for each set partition gives us $\tilde\pi\leq \sigma$. Conversely, given $\tilde\pi\leq \sigma$ we have $\pi\leq \bar \pi =\tilde\pi\oplus  n\leq \sigma\oplus n$.}
We then have 
$$\bx_{\hat 1_{n-1}}\uparrow=\sum_{\pi{\vdash [n]}}\sum_{\tilde \pi \leq \sigma\leq \hat 1_{n-1}} \mu_{\Pi}(\sigma,\hat 1 _{n-1})\bx_{\pi}.$$ By the definition of M\"obius functions the inner sum is 0 unless $\tilde \pi=\hat 1_{n-1}$, so we have finished proving the equation in this lemma since it is not hard to see all $\pi$ with $\tilde \pi=\hat 1 _{n-1}$ is the set of all $\pi\vdash [n]$ with at most two blocks where every block contains $n$ or $n-1$. 

For the last part, given a set partition $\pi$ and a suitable permutation $\delta$ we know that {$f$ is $\bx$-positive if and only if $\delta\circ f$ is $\bx$-positive} by Equation~\eqref{eq:xdelta}. Also, it is not hard to see that  if given a permutation $\delta \in \fS _{n-1}$ with $\delta(n-1)=n-1$, { $\tilde \delta\in\fS_{n}$ that is $\delta$  with $\tilde \delta(n)=n$}  and $f\in \NCSym$ is of homogeneous degree {$n-1$} then 
${\tilde \delta}\circ (f\uparrow)=(\delta\circ f)\uparrow$. 
Say  particularly that for {$\pi\vdash[n-1]$ that} $\delta\in \fS _{n-1}$ is a permutation where $\delta(\pi)={\hat 1_{k_1}|\hat 1_{k_2} |\cdots|\hat 1_{k_L}}$ and $\delta({n-1})={n-1}$. 
 We can  see that by Equation~\eqref{eq:xdelta} and Lemma~\ref{lem:xprod}
 $${\delta(\bx_{\pi})\uparrow=}\bx_{\delta(\pi)}\uparrow={(\bx_{\hat 1_{k_1}}\bx_{\hat 1_{k_2}}\cdots \bx_{\hat 1_{k_L}})\uparrow=\bx_{\hat 1_{k_1}}\bx_{\hat 1_{k_2}}\cdots \bx_{\hat 1_{k_{L-1}}}(\bx_{\hat 1_{k_L}}\uparrow)}$$ is $\bx$-positive by Lemma~\ref{lem:xprod} and by what we have already shown. To complete the proof note that 
 $$\tilde \delta^{-1}\circ (\delta(\bx_{\pi})\uparrow)=\tilde \delta^{-1}\circ \tilde\delta\circ (\bx_{\pi}\uparrow)=\bx_{\pi}\uparrow.$$\end{proof}

We can now give our desired classification for $\bx$-positivity and $\bx$-negativity.

\begin{theorem}\label{the:xpos}
For a connected graph $G$ with distinct vertex labels in $[n]$ we have $$Y_G=(-1)^{n-1}Z_G$$
where $Z_G$ is $\bx$-positive. Consequently, for  
a graph $G$ with $k$ connected components and distinct vertex labels in $[n]$ we have $$Y_G=(-1)^{n-k}Z_G$$
where $Z_G$ is $\bx$-positive. Therefore $Y_G$ is $\bx$-positive if $n-k$ is even, and is $\bx$-negative if $n-k$ is odd.
\end{theorem}

\begin{proof}
We will prove this for a connected graph by doing an induction on the number of vertices $n$ and the number of edges $m$. Our base case is when $n=1$ in which case $Y_G=p_1=\bx_1$. 

We now assume our result is true for any connected graph with vertices $0<|V|<n$. We  already  have the result in the case that a connected graph has the minimal number of edges, that is, a tree by Lemma~\ref{lem:trees}, so let $m>n-1$.  We assume the result is true for a connected graph on $n$ vertices with number of edges $n-1< |E|<m$. 

Since we know that $m>n-1$ we know that $G$ has an edge that is not a bridge, that is, an edge $\epsilon$ such that $G-\epsilon$ is connected. Without loss of generality, by the Relabelling Proposition (Proposition~\ref{prop:relabeling}) assume that $\epsilon$ is the edge between vertices $n-1$ and $n$. By Deletion-Contraction (Proposition~\ref{prop:DeletionContraction}) we have 
$$Y_G=Y_{G-\epsilon}-Y_{G/\epsilon}\uparrow.$$
By induction we know that $Y_{G-\epsilon}=(-1)^{n-1}Z_{G-\epsilon}$ where $Z_{G-\epsilon}$ is $\bx$-positive and $Y_{G/\epsilon}=(-1)^{n-2}Z_{G/\epsilon}$ where $Z_{G/\epsilon}$ is $\bx$-positive. By Lemma~\ref{lem:x_induce} we know  {an induced} $\bx$-positive function is $\bx$-positive, so  $Y_{G/\epsilon}\uparrow=(-1)^{n-2}\tilde Z_{G/\epsilon}$ where $\tilde Z_{G/\epsilon}$ is $\bx$-positive.  Hence the result follows for a connected graph.

The result for a graph with $k$ connected components now follows by Corollary~\ref{cor:Y_G}, Lemma~\ref{lem:xprod} and the result for a connected graph.
\end{proof}

Having classified when $Y_G$ is $\bx$-positive and $\bx$-negative, we now address $e$-positivity and $e$-negativity. For this we will need the following three lemmas, the first of which also has an interesting corollary, and the second of which yields a natural multiplicative version of the fundamental theorem of symmetric functions in $\NCSym$.

\begin{lemma}\label{lem:YGise} For a set partition $\pi$ we have
$$Y_{K_\pi}=e_\pi .$$
\end{lemma}

\begin{proof} This follows immediately by comparing the definitions of $Y_{K_\pi}$ and $e_\pi .$
\end{proof}

We now have an elegant relationship between the $e$-basis and the $\bx$-basis.

\begin{corollary}\label{cor:eisxpos} For a set partition $\pi = B_1/B_2/\cdots /B_{\ell(\pi)}\vdash [n]$ we have  $e_\pi $ is $\bx$-positive if $n-\ell(\pi)$ is even and is $\bx$-negative if $n-\ell(\pi)$ is odd.
\end{corollary}

\begin{proof} This follows immediately by Theorem~\ref{the:xpos} and Lemma~\ref{lem:YGise}.
\end{proof}

For our second lemma we show that the $e$-basis satisfies the property exhibited by the $p$-basis in Lemma~\ref{lem:powersum}, and the $\bx$-basis in Lemma~\ref{lem:xprod}. This was also shown by the first author using M\"obius functions \cite[Lemma 2.1]{Dladders}.

\begin{lemma}\label{lem:eprod}For two set partitions $\pi$ and $\sigma$ we have
$$e_{\pi \slashp \sigma}=e_{\pi}e_{\sigma}.$$
\end{lemma}

\begin{proof} First note that
\begin{equation}\label{eq:kprod}
K_\pi\slashp K_\sigma = K_{\pi \slashp \sigma}. \end{equation} Consequently by Lemma~\ref{lem:YGise} and Proposition~\ref{prop:graph_mult} we have
$$e_{\pi \slashp \sigma}=Y_{K_{\pi\slashp \sigma}}= Y_{K_{\pi}\slashp K_{\sigma}}= Y_{K_{\pi} }Y_{K_{\sigma}}=e_{\pi}e_{\sigma}.$$\end{proof}

We now use this second lemma to establish for $\NCSym$ the version of the fundamental theorem of symmetric functions that states that the elementary symmetric functions $\{e_i \suchthat i\geq 1\}$ in $\Sym$ are algebraically independent over $\bQ$ and generate $\Sym$.

\begin{theorem}\label{the:fundamentaltheorem} The elementary symmetric functions $\{e_\alpha \suchthat \alpha \mbox{ is atomic}\}$  in $\NCSym$  are algebraically independent over $\bQ$ and freely generate $\NCSym$.
\end{theorem}

\begin{proof}
Since $\{ e_\pi\suchthat\pi\vdash [n], n\geq 1 \}$ is a basis for $\NCSym$  they are all linearly independent. They are also multiplicative along atomic decompositions since Lemma~\ref{lem:eprod} shows that for $\pi=\alpha_1\slashp \alpha_2 \slashp \cdots\slashp \alpha_k$, written  as its unique atomic decomposition, we have 
$$e_{\pi}=\prod_{i=1}^k e_{{\alpha_i}}.$$
Thus,  all $\{e_\alpha \suchthat \alpha \mbox{ is atomic}\}$ are algebraically independent over $\bQ$ and freely generate $\NCSym$.
\end{proof}

For our third lemma recall that if some function $f$ in a vector space is expressed in terms of a basis $\{ b_i \}$ then $[b_i]f$ denotes the coefficient of basis element $b_i$ in $f$. We will also need the following change of basis formula in $\NCSym$ \cite[Theorem 3.4]{RS} that says that for a set partition $\pi \vdash [n]$
\begin{equation}p_{\pi}=\frac{1}{\mupi(\minel, \pi)}\sum_{\sigma\leq \pi}\mupi(\sigma,\pi)e_{\sigma}.
\label{eq:p_to_e}
\end{equation}

\begin{lemma}\label{lem:coeff}
Given a set partition $\pi ={B_1/B_2}\vdash [n]$ with two blocks and a graph $G$ with distinct vertex labels in $[n]$,  we have
$$[e_{\maxel}]Y_G=\frac{1}{(n-1)!}|[p_{\maxel}]Y_G|$$
and 
$$[e_{B_1/B_2}]Y_G=-\frac{1}{(n-1)!}|[p_{\maxel}]Y_G|+\frac{(-1)^n}{(|B_1|-1)!(|B_2|-1)!}\sum_{S\subseteq E, \pi(S)=B_1/B_2}(-1)^{|S|}.$$
\end{lemma}

\begin{proof}
Using Proposition~\ref{prop:gebhard_sagan} and Equation~\eqref{eq:p_to_e} we have 
\begin{align*}
Y_G&=\sum_{S\subseteq E(G)}\sum_{\sigma\leq \pi(S)}(-1)^{|S|}\frac{\mupi(\sigma,\pi(S))}{\mupi(\minel, \pi(S))}e_{\sigma}\\
&=\sum_{\sigma}\sum_{S\subseteq E(G), \sigma\leq \pi(S)}(-1)^{|S|}\frac{\mupi(\sigma,\pi(S))}{\mupi(\minel, \pi(S))}e_{\sigma}.
\end{align*}
This means that by Equation~\eqref{eq:0to1}
\begin{align*}
[e_{\maxel}]Y_G&=\sum_{S\subseteq E(G), \maxel = \pi(S)}(-1)^{|S|}\frac{\mupi(\maxel,\pi(S))}{\mupi(\minel, \pi(S))}\\
&=\frac{1}{(-1)^{n-1}(n-1)!}\sum_{S\subseteq E(G), \maxel = \pi(S)}(-1)^{|S|}.
\end{align*}
Note that the summation above is equal to  $[p_{\maxel}]Y_G$, which has sign $(-1)^{n-1}$ by Theorem~\ref{the:MobiusStan} and Proposition~\ref{prop:gebhard_sagan}. This completes the proof for the coefficient $[e_{\maxel}]Y_G$. 

For a set partition $B_1/B_2$ with two blocks we have by Equations~\eqref{eq:0to1} and \eqref{eq:0topi}
\begin{align*}
[e_{B_1/B_2}]Y_G&=\sum_{S\subseteq E(G), B_1/B_2 \leq \pi(S)}(-1)^{|S|}\frac{\mupi(B_1/B_2,\pi(S))}{\mupi(\minel, \pi(S))}\\
&=\sum_{S\subseteq E(G), \maxel= \pi(S)}(-1)^{|S|}\frac{\mupi(B_1/B_2,\pi(S))}{\mupi(\minel, \pi(S))}\\
&+\sum_{S\subseteq E(G), B_1/B_2 = \pi(S)}(-1)^{|S|}\frac{\mupi(B_1/B_2,\pi(S))}{\mupi(\minel, \pi(S))}\\
&=\frac{-1}{(-1)^{n-1}(n-1)!}\sum_{S\subseteq E(G), \maxel= \pi(S)}(-1)^{|S|}\\
&+\frac{1}{(-1)^{n-2}(|B_1|-1)!(|B_2|-1)!}\sum_{S\subseteq E(G), B_1/B_2 = \pi(S)}(-1)^{|S|}.
\end{align*}
Again note that the first summation in the last equality above is equal to $[p_{\maxel}]Y_G$, which has sign $(-1)^{n-1}$ by Theorem~\ref{the:MobiusStan} and Proposition~\ref{prop:gebhard_sagan}. This completes the proof for the coefficient $[e_{B_1/B_2}]Y_G$. 
\end{proof}

We are now ready to classify when $Y_G$ is $e$-positive or $e$-negative.

\begin{theorem} \label{the:epos}
Given a graph $G$ with distinct vertex labels in $[n]$, $Y_G$ is $e$-positive if and only if $G$ is a disjoint {union} of complete graphs. $Y_G$ is never $e$-negative.
\end{theorem}

\begin{proof}By Lemma~\ref{lem:YGise} we know that $Y_{K_{\maxel}}=e_{\maxel}$, which is certainly $e$-positive. We now show   that if $G$ is connected and not a complete graph then $Y_G$ has both a positive and negative term in the elementary basis.

First by combining Theorem~\ref{the:MobiusStan}  and Lemma~\ref{lem:coeff} we know that $[e_{\maxel}]Y_G$ is nonzero and furthermore $[e_{\maxel}]Y_G>0$ for all graphs $G$ with $n$ vertices.

Now assume that $G$ is connected and  has two vertices $u$ and $v$ with no edge between them. Consider the set partition $\pi=B_1/B_2$ where $B_1=\{u,v\}$ and $B_2=[n]\setminus B_1$. Because $G$ lacks an edge from $u$ to $v$ we know that $\pi$ is not a connected partition of $G$ and there does not exist a $S\subseteq E(G)$ such that $\pi(S)=\pi$. Using Theorem~\ref{the:MobiusStan} and Lemma~\ref{lem:coeff} we  see that $[e_{\pi}]Y_G$ is nonzero and
$$[e_{\pi}]Y_G=-\frac{1}{(n-1)!}|[p_{\maxel}]Y_G|<0.$$

Thus, by Proposition~\ref{prop:graph_mult}, Corollary~\ref{cor:Y_G} and Lemma~\ref{lem:eprod}, the only $e$-positive graphs are disjoint unions of complete graphs. Proposition~\ref{prop:graph_mult}, Corollary~\ref{cor:Y_G} and Lemma~\ref{lem:eprod},  also show that there are no $e$-negative graphs.
\end{proof}

Note that the argument in the third paragraph of the above proof holds more generally for a connected graph $G$ and two-block set partition $\pi=B_1/B_2$ that is not a connected partition of $G$. This observation yields the following simple way to locate some negative coefficients.

\begin{corollary}\label{cor:2Bnoe}For a connected graph $G$, if the two-block set partition $\pi=B_1/B_2$ is not connected  then $[e_\pi]Y_G<0$.  \end{corollary}

\section{Chromatic bases for symmetric functions in $\NCSym$}\label{sec:NCSymbases} In this section we discover a multitude of new bases for $\NCSym$ arising from chromatic symmetric functions in noncommuting variables, and establish that with the exception of  $e_\pi$ for all $\pi\vdash [n], n\geq 1$ and {$\bx _\pi$} for all $\pi\vdash [n], n\geq 1$ with blocks of size 1 and 2, none of the known functions in $\NCSym$ can be realized as $Y_G$ for some graph $G$.

For a graph $G$, define the set partition $\pi(G)$ to be the set partition such that each block corresponds to the vertices in a connected component of $G$. Now for each atomic set partition $\alpha \vdash [n]$ choose a graph $G_\alpha$ with $n$ vertices such that $\pi (G_\alpha)=\alpha$. From this, given a generic $\pi\vdash [n]$ written as its unique atomic decomposition $\pi = \alpha _1 \slashp  \alpha _2 \slashp \cdots \slashp \alpha _k$ define the graph
$$G_\pi = G_{\alpha _1} \slashp  G_{\alpha _2} \slashp \cdots \slashp G_{\alpha _k}.$$

\begin{example} \label{ex:Gpi} Consider the following.

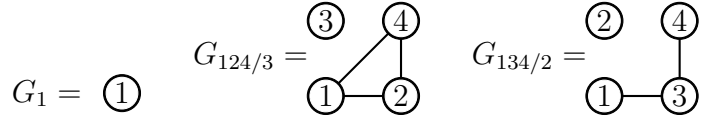
\begin{center}
\begin{tikzpicture}
\coordinate (A) at (0,0);
\filldraw (A) circle (7pt);
\filldraw[white] (A) circle [radius=6pt] node {\textcolor{black}{$1$}};
\draw (-1,0) node {\textcolor{black}{$G_1=$}};
\end{tikzpicture} 
\hspace{.1in}
\begin{tikzpicture}
\coordinate (A) at (0,0);
\coordinate (B) at (1,0);
\coordinate (C) at (1,1);
\coordinate (D) at (0,1);
\draw[thick] (A)--(B)--(C)--(A);
\filldraw (A) circle (7pt);
\filldraw (B) circle (7pt);
\filldraw (C) circle (7pt);
\filldraw (D) circle (7pt);
\filldraw[white] (A) circle [radius=6pt] node {\textcolor{black}{$1$}};
\filldraw[white] (B) circle [radius=6pt] node {\textcolor{black}{$2$}};
\filldraw[white] (C) circle [radius=6pt] node {\textcolor{black}{$4$}};
\filldraw[white] (D) circle [radius=6pt] node {\textcolor{black}{$3$}};
\draw (-1,.5) node {\textcolor{black}{$G_{124/3}=$}};
\end{tikzpicture} 
\hspace{.1in}
\begin{tikzpicture}
\coordinate (A) at (0,0);
\coordinate (B) at (1,0);
\coordinate (C) at (1,1);
\coordinate (D) at (0,1);
\draw[thick] (A)--(B)--(C);
\filldraw (A) circle (7pt);
\filldraw (B) circle (7pt);
\filldraw (C) circle (7pt);
\filldraw (D) circle (7pt);
\filldraw[white] (A) circle [radius=6pt] node {\textcolor{black}{$1$}};
\filldraw[white] (B) circle [radius=6pt] node {\textcolor{black}{$3$}};
\filldraw[white] (C) circle [radius=6pt] node {\textcolor{black}{$4$}};
\filldraw[white] (D) circle [radius=6pt] node {\textcolor{black}{$2$}};
\draw (-1,.5) node {\textcolor{black}{$G_{134/2}=$}};
\end{tikzpicture} 
\end{center}

For the above note that $\pi(G_1) =1$, $\pi (G_{124/3})= 124/3$ and $\pi(G_{134/2})=134/2$. Given these graph choices for atomic partitions $1$, $124/3$, $134/2$ we get the following graph associated to $134/2/5/679/8 = 134/2\slashp 1 \slashp 124/3$. 
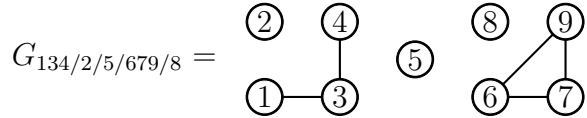
\begin{center}
\begin{tikzpicture}
\coordinate (A) at (0,0);
\coordinate (B) at (1,0);
\coordinate (C) at (1,1);
\coordinate (D) at (0,1);
\coordinate (E) at (2,.5);
\coordinate (F) at (3,0);
\coordinate (G) at (4,0);
\coordinate (H) at (4,1);
\coordinate (I) at (3,1);
\draw[thick] (F)--(G)--(H)--(F);
\filldraw (F) circle (7pt);
\filldraw (G) circle (7pt);
\filldraw (H) circle (7pt);
\filldraw (I) circle (7pt);
\filldraw[white] (F) circle [radius=6pt] node {\textcolor{black}{$6$}};
\filldraw[white] (G) circle [radius=6pt] node {\textcolor{black}{$7$}};
\filldraw[white] (H) circle [radius=6pt] node {\textcolor{black}{$9$}};
\filldraw[white] (I) circle [radius=6pt] node {\textcolor{black}{$8$}};
\draw[thick] (A)--(B)--(C);
\filldraw (A) circle (7pt);
\filldraw (B) circle (7pt);
\filldraw (C) circle (7pt);
\filldraw (D) circle (7pt);
\filldraw (E) circle (7pt);
\filldraw[white] (A) circle [radius=6pt] node {\textcolor{black}{$1$}};
\filldraw[white] (B) circle [radius=6pt] node {\textcolor{black}{$3$}};
\filldraw[white] (C) circle [radius=6pt] node {\textcolor{black}{$4$}};
\filldraw[white] (D) circle [radius=6pt] node {\textcolor{black}{$2$}};
\filldraw[white] (E) circle [radius=6pt] node {\textcolor{black}{$5$}};
\draw (-2,.5) node {\textcolor{black}{$G_{134/2/5/679/8}=$}};
\end{tikzpicture} 
\end{center}
\end{example}

\begin{theorem}\label{the:independent}
Let $\{G_{\alpha}:\alpha\text{ is atomic}\}$ be a set of graphs such that $G_{\alpha}$  has $n$ vertices and $\pi(G_{\alpha})=\alpha$ for each $\alpha$. Then 
$$\{Y_{G_{\pi}}:\pi\vdash [n]\}$$
is a $\bQ$-basis for $\NCSym^n$. Additionally,   all $\{Y_{G_{\alpha}}:\alpha\text{ is atomic}\}$ are algebraically independent over $\bQ$ and freely generate $\NCSym$. 
\end{theorem}

\begin{proof}
Choose a set $\{G_{\alpha}:\alpha\text{ is atomic}\}$  such that $G_{\alpha}$  has $n$ vertices and $\pi(G_{\alpha})=\alpha$ for each $\alpha$. Given  $\pi\vdash[n]$  we know  by definition that if $\sigma \in L_{G_{\pi}}$ then $\sigma\leq \pi$. Hence, by Theorem~\ref{the:MobiusStan}, 
$$Y_{G_{\pi}}=\sum_{\sigma\leq \pi}c_{\sigma\pi}p_{\sigma}$$and $c_{\pi\pi}=\mul(\hat{0},\pi)\neq 0$ also from Theorem~\ref{the:MobiusStan}. From this we can conclude that the $Y_{G_\pi}$ for all $\pi\vdash[n]$ are linearly independent and hence $\{Y_{G_{\pi}}:\pi\vdash [n]\}$ is a $\bQ$-basis for $\NCSym^n$. 
 They are also multiplicative along atomic decompositions since Proposition~\ref{prop:graph_mult} shows that for $\pi=\alpha_1\slashp \alpha_2\slashp \cdots \slashp \alpha_k$, written  as its unique atomic decomposition, we have 
$$Y_{G_{\pi}}=\prod_{i=1}^k Y_{G_{\alpha_i}}.$$
Thus,  all $\{Y_{G_{\alpha}}:\alpha\text{ is atomic}\}$ are algebraically independent over $\bQ$ and freely generate NCSym. 
\end{proof}

Now we move on to identifying whether $Y_G$, for some graph $G$, is a known function in $\NCSym$. The first functions we will study are the Schur-like functions of Rosas and Sagan, $S_\lambda$ where $\lambda$ is an \emph{integer} partition. They are defined in \cite[Section 6]{RS}, and since we will need them only for the next result, we refer the reader there for the general definition, though note that $S_{(1^n)}=e_{\maxel}$, where $(1^n)$ is the integer partition consisting of $n$ parts equal to 1. The only other property we recall is that for $\lambda \vdash n$ \cite[Theorem 6.2 (iii)]{RS}
\begin{equation}\label{eq:rhoS}\rho(S_\lambda)=n!s_\lambda\end{equation}where $s_\lambda$ is the classical Schur function.

\begin{proposition}\label{prop:YGasS} Of the functions $\{S_\lambda \} _{\lambda \vdash n\geq 1}$ in $\NCSym$ and their scalar multiples only $S_{(1^n)}$ can be realized as $Y_G$ for some graph $G$. In particular
$$S_{(1^n)}=Y_{K_n}.$$
\end{proposition}

\begin{proof} Assume $Y_G=c_\lambda S_\lambda$. By Equation~\eqref{eq:rhoS} we have 
$$X_G=\rho(Y_G)=\rho(c_\lambda S_\lambda)=c_\lambda n! s_\lambda$$that by \cite[Theorem 2.8]{ChovW2} is possible if and only if $\lambda = (1^n)$, $G=K_n$   and $c_\lambda = 1$. 
\end{proof}

For the functions indexed by set partitions we have one more family to define. The \emph{complete homogeneous symmetric function in $\NCSym$}, $h_\pi$ where $\pi \vdash [n]$, is given by \cite[Theorem 3.4]{RS}
$$h_\pi = \sum _{\sigma \leq \pi} |\mupi (\minel , \sigma)|p_\sigma .$$Due to the common proof technique, we will now consider all except the $\bx$-basis.

\begin{theorem}\label{the:YGasothers}
Of the functions $\{e_\pi \} _{\pi \vdash [n], n\geq 1}$, $\{h_\pi \} _{\pi \vdash [n], n\geq 1}$, $\{m_\pi \} _{\pi \vdash [n], n\geq 1}$ and $\{p_\pi \} _{\pi \vdash [n], n\geq 1}$ in $\NCSym$ and their scalar multiples only $\{e_\pi \} _{\pi \vdash [n], n\geq 1}$ can be realized as $Y_G$ for some graph $G$. In particular
$$e_\pi=Y_{K_\pi}.$$
\end{theorem}

\begin{proof} Since $e_{\minel}=h_{\minel}=p_{\minel}$ for $n\geq 1$ we consider this to be $e_{\minel}$ for convenience, and likewise consider $e_{\maxel}= m_{\minel}$ for $n\geq 1$ to be $e_{\maxel}$. 
Assume $Y_G=c_\pi b_\pi$ where $b=h,m$ or $p$ and $\pi \neq \minel$ for $n\geq 1$. Then
$$X_G=\rho(Y_G)=\rho(c_\pi b_\pi)=c_\pi \rho(b_\pi).$$Now \cite[Theorem 2.1]{RS} states that $\rho(b_\pi)=\lambda (\pi)! b_{\lambda (\pi)}$ if $b=h$, $\rho(b_\pi)=\lambda (\pi)^! b_{\lambda (\pi)}$ if $b=m$, and $\rho(b_\pi)= b_{\lambda (\pi)}$ if $b=p$. However, by \cite[Theorem 2.8]{ChovW2} none of these outcomes  {are} possible. By \cite[Theorem 2.1]{RS} we also have
$$X_G=\rho(Y_G)=\rho(c_\pi e_\pi)=c_\pi \rho(e_\pi) = c_\pi \lambda (\pi)! e_{\lambda (\pi)}$$that by \cite[Theorem 2.8]{ChovW2} is possible if and only if $G$ is a disjoint union of complete graphs and $c_\pi =1$. The final equation is then by Lemma~\ref{lem:YGise}.
\end{proof}

Finally, we consider the $\bx$-basis, for which we need the following lemma.

\begin{lemma} \label{lem:x_twoterms}
For any graph $G$ containing a connected component with at least three vertices, we have $Y_G$ has at least two terms in its expansion in the $\bx$-basis. 
\end{lemma}

\begin{proof}Consider a connected graph $G$ with $n\geq 3$ vertices. Substituting  Equation~\eqref{eq:pasx} into Theorem~\ref{the:MobiusStan}  we get 
$$Y_G=\sum_{\pi \in L_G}\mul(\minel, \pi)\sum_{\sigma\leq \pi }\bx_{\sigma}$$and $\maxel\in L_G$
so the coefficient of $\bx_{\maxel}$ is $\mul(\minel,\maxel)\neq 0$. If $G$ is not the complete graph then there exist two vertices that are not adjacent. Say their labels are $u$ and $v$ with $u<v$. Consider the set partition 
$$\sigma=uv/1\cdots \hat u \cdots \hat v \cdots n \notin L_G$$where $\hat u, \hat v$ means that $u$ and $v$ are removed. Note the coefficient of $\bx_{\sigma}$ is also $\mul(\minel,\maxel)\neq 0$, and we have two nonzero terms. 
If $G$ is the complete graph, then the poset $L_G=\Pi_n$, and using Equations~\eqref{eq:0to1} and \eqref{eq:0topi} we can calculate the coefficient of $\pi=1/23\cdots n$  is 
\begin{align*}\mu_{\Pi}(\minel, \pi)+\mu_{\Pi}(\minel,\maxel)
&=(-1)^{n-2}(n-2)!+ (-1)^{n-1}(n-1)!\\
&=((-1)^{n-2}+ (-1)^{n-1}(n-1))(n-2)!\\
&=(-1)^{n-1}(n-2)(n-2)!\neq 0\end{align*} since $n>2$. 

Lastly, consider a graph $G$ with a connected component with at least three vertices, $\tilde{G}$. By Corollary~\ref{cor:Y_G} and Lemma~\ref{lem:xprod} we know that $Y_G$ has at least as many terms as $Y_{\tilde{G}}$ and we are done. 
\end{proof}

Our final theorem completes the classification of which known functions in $\NCSym$ can be realized as $Y_G$ for some graph $G$.

\begin{theorem}\label{the:YGasx} Of the functions $\{\bx_\pi \} _{\pi \vdash [n], n\geq 1}$ in $\NCSym$ and their scalar multiples only $\pm \bx _\pi$ where $\pi = B_1/B_2/\cdots /B_{\ell(\pi)}$ and $|B_i| = 1, 2$ for $1\leq i \leq \ell(\pi)$ can be realized as $Y_G$ for some graph $G$. In particular
$$\bx_\pi=(-1)^{t_\pi}Y_{K_\pi}$$where $t_\pi = $ the number of $|B_i|=2$ for $1\leq i \leq \ell(\pi)$.
\end{theorem}

\begin{proof} Assume $Y_G=c_\pi\bx_\pi$ for some $\pi \vdash [n], n\geq 1$. By Lemma~\ref{lem:x_twoterms} we know $G$ only contains connected components with one or two vertices. Therefore $G=K_\pi$ where $\pi = B_1/B_2/\cdots /B_{\ell(\pi)}$ and $|B_i| = 1, 2$ for $1\leq i \leq \ell(\pi)$. By using Lemma~\ref{lem:YGise}, Proposition~\ref{prop:gebhard_sagan} and Equation~\eqref{eq:xasp}, respectively, twice we   have 
\begin{align*}
e_1&=Y_{K_1}=p_1=\bx_1\\
e_{12}&=Y_{K_{12}}=-p_{12}+p_{1/2}=-\bx _{12}.
\end{align*}
Thus by choosing $\delta \in \fS _n$ such that $$\delta(K_\pi) = K_{|B_1|}\slashp K_{|B_2|}\slashp \cdots \slashp K_{|B_{\ell(\pi)}|}= K_{\hat 1_{|B_1|}}\slashp K_{\hat 1_{|B_2|}}\slashp \cdots \slashp K_{\hat 1_{|B_{\ell(\pi)}|}}$$we have the following equalities, given respectively by Corollary~\ref{cor:Y_G}, Lemma~\ref{lem:YGise}, the above substitutions to the $\bx$-basis, Lemma~\ref{lem:xprod}, Equation~\eqref{eq:xdelta} and the definition of $\delta$
\begin{align*}
Y_{K_\pi}&= \delta ^{-1} \circ (Y_{K_{|B_1|}} Y_{K_{|B_2|}}  \cdots   Y_{K_{|B_{\ell(\pi)}|}})= \delta ^{-1} \circ (e_{\hat 1_{|B_1|}} e_{\hat 1_{|B_2|}}  \cdots   e_{\hat 1_{|B_{\ell(\pi)}|}})\\
&= \delta ^{-1} \circ ((-1)^{t_\pi}\bx_{\hat 1_{|B_1|}} \bx_{\hat 1_{|B_2|}}  \cdots   \bx_{\hat 1_{|B_{\ell(\pi)}|}})= (-1)^{t_\pi}\delta ^{-1} \circ (\bx_{\hat 1_{|B_1|}\slashp \hat 1_{|B_2|} \slashp \cdots   \slashp \hat 1_{|B_{\ell(\pi)}|}})\\
&= (-1)^{t_\pi}  (\bx_{\delta ^{-1}(\hat 1_{|B_1|} \slashp \hat 1_{|B_2|} \slashp  \cdots \slashp  \hat 1_{|B_{\ell(\pi)}|})})= (-1)^{t_\pi}  \bx_{\pi}\\
\end{align*}where $t_\pi = $ the number of $|B_i|=2$ for $1\leq i \leq \ell(\pi)$.
\end{proof}

\begin{example}\label{ex:YGasx2} Returning to Example~\ref{ex:YG} we have  $Y_{K_{13/2}}=-\bx _{{13/2}}$.
\end{example}

\section*{Acknowledgements}\label{ack} {We would like to thank the referee for the care they took with our paper.}


\bibliographystyle{plain}

\begin{thebibliography}{10}

\bibitem{28authors} M.~Aguiar et al, Supercharacters, symmetric functions in noncommuting variables, and related {H}opf algebras, \emph{Adv. Math.}  229,  2310--2337 (2012). 

\bibitem{BHRZ} N.~Bergeron, C.~Hohlweg, M.~Rosas and M.~Zabrocki, {G}rothendieck bialgebras, partition lattices, and symmetric functions in noncommutative variables, \emph{Electron. J. Combin.} 13, R75 19pp (2006).

\bibitem{BRRZ} N.~Bergeron, C.~Reutenauer, M.~Rosas and M.~Zabrocki, Invariants and coinvariants of the symmetric group in noncommuting variables, \emph{Canad. J. Math.} 60, 266--296 (2008).

\bibitem{BZ} N.~Bergeron and M.~Zabrocki, The  {H}opf  algebras  of  symmetric  functions  and  quasi-symmetric  functions  in  non-commutative variables are free and co-free, \emph{J. Algebra Appl.} 8,  581--600 (2009).

\bibitem{Birk}  G.~Birkhoff, A determinant formula for the number of ways of coloring a map, \emph{Ann. of Math.} 14, 42--46
(1912).

\bibitem{CanSagan} M.~Can and B.~Sagan, Partitions, rooks, and symmetric functions in noncommuting variables, \emph{Electron. J. Combin.} 18, P3 7pp (2011). 


\bibitem{ChoHuh} S.~Cho and J.~Huh, On $e$-positivity and $e$-unimodality
of chromatic quasisymmetric functions, {\tt arXiv:1711.07152}

\bibitem{ChovW} S.~Cho and S.~van Willigenburg, Chromatic bases for symmetric functions, \emph{Electron. J. Combin.} 23, P1.15 6pp (2016).

\bibitem{ChovW2} S.~Cho and S.~van Willigenburg, Chromatic classical symmetric functions,  \emph{J. Comb.} 9, 401--409 (2018).

\bibitem{Dladders}
{S.~Dahlberg, Triangular ladders $P_{d,2}$ are $e$-positive, {\tt arXiv:1811.04885v1}}

\bibitem{lollipop}
S.~Dahlberg and S.~van Willigenburg, Lollipop and lariat symmetric functions, \emph{SIAM J. Discrete Math.} 32, 1029--1039 (2018).

\bibitem{Foley}
A.~Foley, C.~Ho\`{a}ng and O.~Merkel, Classes of graphs with $e$-positive chromatic symmetric function, {\tt arXiv:1808.03391v1}
 
\bibitem{FoleyKin}
{A.~Foley, J.~Kazdan, L.~Kr\"{o}ll, S.~Mart\'{i}nez Alberga, O.~Melnyk and  A.~Tenenbaum, Spiders and their kin ($K_n$), {\tt arXiv:1812.03476v1}}

\bibitem{Gasharov} V.~Gasharov, Incomparability graphs of $(3 + 1)$-free posets are
$s$-positive, \emph{Discrete Math.} 157, 193--197 (1996).

\bibitem{Gash} V.~Gasharov, On {S}tanley's chromatic symmetric function and clawfree graphs, \emph{Discrete Math.} 205, 229--234 (1999).

\bibitem{GebSag} D.~Gebhard and B.~Sagan,  A chromatic symmetric function in noncommuting variables, \emph{J. Algebraic Combin.} 13, 227--255 (2001). 

\bibitem{GP} M.~Guay-Paquet, A modular law for the chromatic symmetric functions of $(3+1)$-free posets, {\tt arXiv:1306.2400v1}

\bibitem{Hamel}
A.~Hamel, C.~Ho\`{a}ng and J.~Tuero,  Chromatic symmetric functions and ${H}$-free graphs, {\tt arXiv:1709.03354v1} 

\bibitem{MM} M.~Harada and M.~Precup, The cohomology of abelian {H}essenberg
varieties and the Stanley-Stembridge conjecture, {\tt arXiv:1709.06736v2}


\bibitem{HuhNamYoo}
{J.~Huh, S.-Y.~Nam and M.~Yoo, Melting lollipop chromatic quasisymmetric functions and {S}chur expansion of unicellular {LLT} polynomials, {\tt arXiv:1812.03445v1}}

\bibitem{LauveM} A.~Lauve and M.~Mastnak, The primitives and antipode in the {H}opf algebra of symmetric functions in noncommuting variables, \emph{Adv. Appl. Math.} 47,  536--544 (2011). 

\bibitem{Orellana} R.~Orellana and G.~Scott, Graphs with equal chromatic symmetric function, \emph{Discrete Math.} 320, 1--14 (2014).

\bibitem{Paw} 
B.~Pawlowski, Chromatic symmetric functions via the group algebra of $S_n$, {\tt arXiv:1802.05470v3}

\bibitem{RS} M.~Rosas and B.~Sagan, Symmetric functions in noncommuting variables, \emph{Trans. Amer. Math. Soc.} 358,  215--232 (2006).

\bibitem{SW} J.~Shareshian and M.~Wachs, Chromatic quasisymmetric functions, \emph{Adv. Math.} 295, 497--551 (2016).

\bibitem{Stan95} R.~Stanley, A symmetric function generalization of the chromatic polynomial of a graph, \emph{Adv. Math.} 111, 166--194 (1995).

\bibitem{Stanley2} R.~Stanley, Graph colorings and related symmetric functions: ideas and applications. A description of results, interesting applications, \& notable open problems, \emph{Discrete Math.} 193, 267--286 (1998).

\bibitem{StanStem} R.~Stanley and J.~Stembridge, On immanants of {J}acobi-{T}rudi matrices and permutations with restricted position, \emph{J. Combin. Theory Ser. A} 62, 261--279 (1993).

\bibitem{Thiem} N.~Thiem, Branching rules in the ring of superclass functions of unipotent upper-triangular matrices, \emph{J. Algebraic
Combin.} 31, 267--298 (2010).

\bibitem{Tsujie} 
S.~Tsujie,  The chromatic symmetric functions of trivially perfect graphs and cographs, {\tt  arXiv:1707.04058v1}

\bibitem{Wolf} M.~Wolf, Symmetric functions of non-commutative elements, \emph{Duke Math. J.}  2, 626--637 (1936).

\bibitem{Wolfe} M.~Wolfe, Symmetric chromatic functions, \emph{Pi Mu Epsilon Journal} 10, 643--757 (1998).


\end{thebibliography}

\def\cprime{$'$}

\end{document}